\documentclass[journal]{IEEEtran}

\usepackage{amsmath,amssymb,amsthm,mathtools}
\usepackage{graphicx}
\usepackage{booktabs}
\usepackage[colorlinks,citecolor=blue,urlcolor=blue,linkcolor=blue]{hyperref}
\usepackage{cite}

\graphicspath{{figures/}}

\theoremstyle{plain}
\newtheorem{theorem}{Theorem}
\newtheorem{proposition}[theorem]{Proposition}
\newtheorem{lemma}[theorem]{Lemma}
\newtheorem{corollary}[theorem]{Corollary}

\theoremstyle{definition}
\newtheorem{definition}[theorem]{Definition}
\newtheorem{assumption}[theorem]{Assumption}

\newcommand{\Om}{\Omega}
\newcommand{\Omext}{\Omega_e}
\newcommand{\Ball}{B_R}
\newcommand{\Sphere}{\mathbb{S}^{n-1}}
\newcommand{\Gam}{\Gamma}
\newcommand{\eps}{\varepsilon}
\newcommand{\R}{\mathbb{R}}
\newcommand{\F}{\mathcal{F}}
\newcommand{\Kop}{\mathcal{K}}
\newcommand{\Lop}{\mathcal{L}}
\newcommand{\concat}{\mathcal{J}}
\newcommand{\Dom}{D}

\title{Output-Feedback Boundary Control of Reaction--Diffusion PDEs on Arbitrary Lipschitz Domains: A Target-Domain Approach}

\author{Rafael~Vazquez%
\thanks{R.~Vazquez is with the Departamento de Ingenier\'{\i}a Aeroespacial y Mec\'anica de Fluidos, Universidad de Sevilla, Camino de los Descubrimientos s.n., 41092 Sevilla, Spain (e-mail: rvazquez1@us.es). The author acknowledges support of grant PID2023-147623OB-I00 funded by MICIU/AEI/ 10.13039/501100011033 and by ``ERDF A way of making Europe''.}}

\begin{document}

\maketitle

\begin{abstract}
We present a domain-extension framework for output-feedback boundary stabilization of reaction-diffusion equations on arbitrary bounded Lipschitz domains, including non-convex and multiply connected geometries. The plant is posed on an irregular domain whose boundary has actuated and uncontrolled portions. Just as backstepping transforms the plant dynamics into a stable target system, the method embeds the plant in a \emph{target domain}, such as a ball or a rectangle, where a stabilizing design is already known. Every boundary portion through which the extension proceeds must carry actuation and the complementary collocated measurement. Uncontrolled portions are allowed when they are shared with the target boundary and have the same boundary-condition type. The gap between the two domains is filled with a virtual copy of the plant dynamics, coupled to the plant through interface conditions, and the concatenated state evolves exactly as the known closed loop on the target domain. Well-posedness and exponential stability of the physical state follow by restriction. The offline design data are inherited from the target design and are closed-form for constant-coefficient plants on balls and rectangles. Online simulation of the virtual PDE has the same computational character as a full-order PDE observer, a standard component of output-feedback designs. A new explicit Neumann-actuated backstepping law on $n$-balls enlarges the available target designs. Output feedback is obtained by lifting the target-domain observer, driven by a collocated interface measurement relayed through the virtual domain. Numerical experiments on star-shaped, horseshoe, and multiply connected domains, with a partitioned plant/controller implementation and a shared-wall cavity, test the designs.
\end{abstract}


\section{Introduction}\label{sec:intro}

Backstepping has become a central tool for boundary control of parabolic partial differential equations (PDEs), beginning with the foundational work of Krstic and Smyshlyaev~\cite{krstic} on one-dimensional reaction-diffusion equations; for a recent comprehensive survey, see~\cite{survey}. The method transforms the plant into a stable target system via a Volterra-type integral operator, reducing controller design to the solution of a kernel PDE. In one spatial dimension this programme is well understood and has been extended to coupled systems~\cite{aamo,simon}, hyperbolic problems~\cite{auriol,long-nonlinear}, and nonlinear settings~\cite{KKK,krstic5,VK-volterra1,VK-volterra2}.

Moving to higher spatial dimensions introduces fundamental challenges, and progress has historically been tied to symmetry. On rectangular domains, Fourier decomposition reduces the problem to a parameterized family of one-dimensional backstepping designs, an idea rooted in the spatial invariance framework of Bamieh et al.~\cite{bamieh2002distributed}. For domains with rotational symmetry (disks, spheres, and $n$-dimensional balls) angular eigenfunction expansions decouple the problem into radial modes, with explicit kernels in terms of modified Bessel functions~\cite{VK-disk,nball,VK-sphere,VK-square-CDC,Qi-square} and series-form kernels when the reaction coefficient varies radially~\cite{scl2023,zhang2024}. Backstepping observers on higher-dimensional domains were investigated early on by Jadachowski et al.~\cite{jadachowski}. Recent work includes an extended abstract on graph-monotone boundaries~\cite{belhadjoudja-graph} and characteristic-based designs for first-order hyperbolic equations in arbitrary dimensions~\cite{belhadjoudja-hyp1,belhadjoudja-hyp2}.

The practical limitation is that most engineering geometries are neither rectangular nor radially symmetric. For general domains, boundary stabilization is classical at the level of existence and of spectral synthesis: Riccati-based and eigenfunction-based feedback laws are available in great generality~\cite{triggiani,Barbu,raymond,munteanu-book,lagnese}, and constructive finite-dimensional observer-based designs, developed to maturity in one dimension~\cite{katz-fridman,lhachemi-prieur}, have recently reached two and three dimensions~\cite{wang-fridman,lmp2025}. All these methods share a structural feature: the synthesis problem is solved numerically on the irregular domain itself (an operator Riccati equation, the domain's unstable eigenstructure, or matrix inequalities built on its eigenbasis) and is solved anew for every change of geometry. The constructive multidimensional results are confined to rectangles~\cite{wang-fridman}, or treat smooth domains under diagonalizability assumptions, with a finite-dimensional observer built from the domain's eigenstructure and driven by in-domain point measurements~\cite{lmp2025}. The higher-dimensional backstepping observers of~\cite{jadachowski} and the flatness and backstepping designs of~\cite{meurer} likewise require separable coordinate domains. To the best of our knowledge, no boundary feedback design whose offline synthesis is entirely closed-form (no Riccati equation, eigenproblem, or kernel computation on the domain itself) has been available for arbitrary domains in dimension two or higher.

This paper provides one. The plant is posed on an arbitrary bounded Lipschitz domain $\Om$, possibly non-convex, possibly with holes; part of its boundary may be uncontrolled, carrying homogeneous Dirichlet or Neumann conditions, and the remaining part is actuated. The idea is to \emph{extend} $\Om$ through the controlled boundaries to a domain on which a stabilizing design is already known. The framework accepts any such design; we instantiate it with the backstepping designs on balls and rectangles recalled in Section~\ref{sec:toolbox}, because their control laws are closed-form. A virtual copy of the plant PDE fills the gap between the two domains, coupled to the physical state through interface conditions enforcing continuity of trace and flux. The concatenated real/virtual state then satisfies \emph{exactly} the known closed loop on the larger domain, and well-posedness and exponential stability of the physical state follow by restriction.

\emph{The target domain.} The construction is a geometric counterpart of the classical backstepping target system. A backstepping transformation keeps the spatial domain and changes the dynamics. Domain extension keeps the dynamics and changes the geometry. We reserve \emph{target domain} for the containing set $\Dom$. The complete object used by the method is a \emph{target-domain configuration}: $\Dom$, its boundary conditions, the actuator and sensor sets, and a stabilizing design for that layout (Definition~\ref{def:target}). Thus, the same ball with Dirichlet and Neumann target designs gives two configurations. The design reads \emph{extend--transform--restrict}: extension handles the geometry, the inner transformation handles the dynamics, and restriction returns the physical state. The interface makes this embedding dynamical as well as geometric: the plant feels exactly the boundary data it would feel as a sub-region of the controlled target, so $\Om$ evolves as if it were a piece of the ball or rectangle, its dynamics shaped by the target geometry. Table~\ref{tab:parallel} summarizes the parallel.

\begin{table}[t]
\caption{The target-domain construction as a geometric counterpart of backstepping.}
\label{tab:parallel}
\centering
\begin{tabular}{lll}
\toprule
 & Backstepping & Domain extension \\
\midrule
Object transformed & dynamics & geometry \\
Plant & $u_t=\eps\Delta u+\lambda u$ & irregular $\Om$ \\
Target & \emph{target system} & \emph{target domain} $\Dom$ \\
Complete object & dynamics + BCs & target-domain config. \\
Map & Volterra transf. & extension \\
Inverse & inverse transf. & restriction \\
Design burden & kernel PDE & interface conditions \\
\bottomrule
\end{tabular}
\end{table}

\emph{Relation to classical extension arguments and fictitious domains.} Extending a domain to prove controllability is a classical device in the open-loop theory of the heat equation: boundary controllability can be deduced from distributed controllability of an extended domain by restriction~\cite{lions-srev,fursikov}. Embedding an irregular domain into a simple shape also underlies fictitious domain methods in numerical analysis~\cite{glowinski}. Here the extension is part of a feedback interconnection, and the coupled state satisfies the target closed loop exactly. The graph-monotone construction reported in~\cite{belhadjoudja-graph} follows a different route and avoids an auxiliary PDE on the class treated there. Domain extension covers arbitrary bounded Lipschitz domains through the full-boundary ball configuration and also lifts the ball observer.

\emph{What is explicit and what is computed.} The framework inherits the offline synthesis of its target-domain configuration. For balls and rectangles, the kernels used here are closed form; no eigenproblem or kernel equation is solved on the irregular domain. A change of geometry changes the controller domain and its numerical mask. Online, the controller solves the virtual PDE. This is not a new type of computational burden: full-order observer-based PDE controllers already integrate a PDE in real time. In state feedback, the virtual continuation is one observer-sized PDE computation. In output feedback, the target observer is added, and the two controller states can be integrated as one coupled block system. Thus the method remains in the same computational class as PDE observer-based control, with a fixed additional factor rather than a new geometry-dependent synthesis problem. We therefore use \emph{closed-form synthesis} rather than \emph{explicit feedback} for the complete controller. The method requires actuation and complementary sensing on the boundary through which the extension proceeds. Shared uncontrolled boundaries reduce that requirement, as Section~\ref{sec:coupled} makes precise.

\emph{The interface decouples actuation types.} When the extension proceeds through an actuated boundary, the interface carries two matching conditions, of trace and of flux, and the plant's own actuation type decides which of the two the controller imposes; the other is measured and fed to the virtual domain. A Dirichlet-actuated plant can thus be served by the Dirichlet ball design, and so can a Neumann-actuated one: the virtual domain converts between actuation types. Matching of types is required only on shared boundaries, where no virtual buffer exists. The compatibility conditions of Section~\ref{sec:coupled} collect these rules.

\emph{Contributions.} First, we set up the domain-extension framework: for a plant on an arbitrary bounded Lipschitz domain and a compatible target-domain configuration, a Sobolev gluing argument and a broken weak formulation give a two-way equivalence with the target closed loop, so that well-posedness and exponential stability transfer to the physical state by restriction, with the target rate and prefactor intact for every embedded shape.

Second, we read off the boundary architecture the method needs: the extended boundary must carry actuation and the collocated measurement, uncontrolled boundaries are admitted only when shared with the target, and the virtual PDE lets a Dirichlet- or Neumann-actuated plant use the same target design.

Third, we lift the target observer to obtain output feedback from the interface measurement alone. Fourth, we derive an explicit Neumann-actuated ball law, enlarging the set of available targets. Finally, we test the designs on star-shaped, horseshoe, multiply connected, and shared-wall geometries with separate plant and controller solvers.

The transfer argument does not require a particular target shape. Once a target-domain configuration is fixed, its stability constants are independent of the shape of the compatible physical domain. The ball is the main realization developed in full: any bounded physical domain can be placed strictly inside a ball, in which case its entire boundary is extended and must be actuated and sensed. Rectangular configurations cover an important partial-boundary case. Straight uncontrolled walls can be shared with the rectangle, leaving only the remaining irregular boundary to be extended.

It must be acknowledged that the domain-extension idea first appeared in~\cite{VK-Milano}, an exploratory article on possible routes for higher-dimensional backstepping. That paper introduced the geometric intuition and sketched the state-feedback construction. It did not formulate the compatibility conditions for extended and shared boundaries, give a weak transmission formulation, prove equivalence and stability transfer, treat output feedback, derive the Neumann-actuated ball law, or test a partitioned plant/controller implementation. The present paper develops those points. Thus~\cite{VK-Milano} is the conceptual precursor, while the analysis, output-feedback architecture, new target design, and numerical study are contributions of this paper.

\emph{Organization.} Section~\ref{sec:toolbox} reviews the backstepping designs on balls and rectangles and derives the Neumann-actuated law. Section~\ref{sec:coupled} poses the plant, defines target domains and compatibility, and forms the coupled system. Section~\ref{sec:reduction} proves equivalence, well-posedness, and stability transfer. Section~\ref{sec:observer} develops output feedback. Section~\ref{sec:numerics} presents the numerical campaign, and Section~\ref{sec:conclusions} concludes.

\section{Backstepping Designs on Canonical Target Domains}\label{sec:toolbox}

Throughout, $n\geq 2$ is the dimension, $\eps>0$ is the diffusion coefficient, $\lambda>0$ the (destabilizing) reaction coefficient, and $c>0$ a design parameter setting the target decay rate. We write $\mu=(\lambda+c)/\eps$. This section is a recap, included so that the paper is self-contained: except for the Neumann-actuated design of Section~\ref{sec:neumann}, which is new, the material is reproduced from~\cite{nball} (with the target damping added) and~\cite{scl2023,zhang2024}.

\subsection{Ball, Dirichlet actuation: plant, target, transformation}\label{sec:ball-recap}

Consider the reaction-diffusion plant on the ball $\Ball=\{x\in\R^n:|x|<R\}$,
\begin{align}
v_t(t,x) &= \eps \Delta v(t,x)+\lambda v(t,x), \qquad x\in \Ball,\ t>0, \label{eq:ball-pde}\\
v(t,x) &= U(t,x), \qquad x\in \partial \Ball,\ t>0, \label{eq:ball-bc}
\end{align}
with full Dirichlet actuation $U$ on $\partial\Ball$. Backstepping maps \eqref{eq:ball-pde}--\eqref{eq:ball-bc} to the exponentially stable target system
\begin{align}
w_t &= \eps \Delta w-cw, && x\in \Ball,\ t>0, \label{eq:target-pde}\\
w &= 0, && x\in \partial \Ball,\ t>0, \label{eq:target-bc}
\end{align}
which satisfies $\|w(t)\|_{L^2(\Ball)}\le e^{-ct}\|w(0)\|_{L^2(\Ball)}$ by a standard energy estimate. The transformation acts modewise in the spherical harmonics expansion $v(t,r\omega)=\sum_{\ell\ge 0}\sum_{m} v_{\ell,m}(t,r)Y_{\ell,m}(\omega)$, $\omega\in\Sphere$: with
\begin{equation}\label{eq:transformation}
w_{\ell,m}(t,r) = v_{\ell,m}(t,r)-\int_0^r K_\ell(r,s)\,v_{\ell,m}(t,s)\,ds,
\end{equation}
the kernel has the closed form~\cite[eq.~(54)]{nball}
\begin{equation}\label{eq:kernel-bessel}
K_\ell(r,s) = -\,\mu\, s\left(\frac{s}{r}\right)^{\!\ell+n-2} \frac{I_1\!\bigl(\sqrt{\mu(r^2 - s^2)}\bigr)}{\sqrt{\mu(r^2 - s^2)}},
\end{equation}
where $I_1$ is the modified Bessel function of the first kind. (The kernels of~\cite{nball} correspond to $c=0$, a zero-reaction target; the substitution $\lambda\mapsto\lambda+c$ used here adds the target damping $-cw$ and changes nothing in the derivations.) The inverse transformation has the same structure with kernel $L_\ell(r,s) = -\mu s (s/r)^{\ell+n-2} J_1(\zeta)/\zeta$, $\zeta=\sqrt{\mu(r^2-s^2)}$, where $J_1$ is the (ordinary) Bessel function~\cite[eq.~(92)]{nball}. We write $\mathcal T=I-\Kop$ and $\mathcal T^{-1}=I+\Lop$ for the corresponding operators on $L^2(\Ball)$. Both are bounded on $L^2(\Ball)$ and $H^1(\Ball)$ with constants independent of the mode index~\cite[Prop.~1]{nball}.

The feedback law enforcing the target boundary condition \eqref{eq:target-bc} is the trace of the transformation,
\begin{equation}\label{eq:ball-feedback}
U = \F[v] := (\Kop v)\big|_{\partial\Ball},
\end{equation}
acting modewise as $U_{\ell,m}=\int_0^R K_\ell(R,s)\,v_{\ell,m}(s)\,ds$. Summing the spherical harmonics produces the law in physical space~\cite{nball}: for $x\in\partial\Ball$,
\begin{equation}\label{eq:physical-law}
U(x) = -\frac{\sqrt{\mu}}{A_{n-1}}\!\int_{\Ball}\! I_1\!\Bigl(\sqrt{\mu(R^2{-}|\xi|^2)}\Bigr) \frac{\sqrt{R^2-|\xi|^2}}{|x-\xi|^{n}}\, v(\xi)\, d\xi,
\end{equation}
where $A_{n-1}=2\pi^{n/2}/\Gamma(n/2)$ is the area of the unit sphere: a single closed-form expression, valid in every dimension, whose kernel is the product of the one-dimensional backstepping kernel (a purely radial factor) and a Poisson-kernel-like factor carrying the geometry of the sphere.

\begin{theorem}[{Ball closed loop, Dirichlet actuation~\cite{nball}}]\label{thm:ball-backstepping}
For every $v_0\in L^2(\Ball)$, the closed loop \eqref{eq:ball-pde}--\eqref{eq:ball-bc}, \eqref{eq:ball-feedback} has a unique weak solution, given by $v=(I+\Lop)w$ with $w$ the standard weak solution of \eqref{eq:target-pde}--\eqref{eq:target-bc} from $w_0=(I-\Kop)v_0$. There is a constant $C_{\mathcal T}\ge 1$, independent of $v_0$, such that
\begin{equation}\label{eq:ball-decay}
\|v(t,\cdot)\|_{L^2(\Ball)} \leq C_{\mathcal T}\, e^{-ct}\|v_0\|_{L^2(\Ball)},\qquad t\ge 0.
\end{equation}
\end{theorem}

For output feedback, the measurement of the ball design is the boundary flux
\begin{equation}\label{eq:ball-measurement}
y(t,\xi) = \partial_{r} v(t,\xi),\qquad \xi\in\partial\Ball,
\end{equation}
and the observer of~\cite[Thm.~2]{nball} copies the plant with output injection,
\begin{align}
\hat v_t &= \eps\Delta\hat v + \lambda\hat v + \mathcal{P}\!\left[y-\partial_{r}\hat v\big|_{\partial\Ball}\right], && x\in \Ball, \label{eq:obs-ball-pde}\\
\hat v &= U, && x\in \partial\Ball, \label{eq:obs-ball-bc}
\end{align}
where $\mathcal P$ acts modewise as multiplication by the explicit gain
\begin{equation}\label{eq:obs-gain}
p_\ell(r) = -(\lambda+c)\, r\left(\frac{r}{R}\right)^{\!\ell-1} \frac{I_1\!\bigl(\sqrt{\mu(R^2-r^2)}\bigr)}{\sqrt{\mu(R^2-r^2)}},
\end{equation}
obtained as $p_\ell=\eps P_\ell(\cdot,R)$ from the observer kernel of~\cite{nball} (the exponent $\ell-1$ in \eqref{eq:obs-gain} does not depend on the dimension). The output-feedback closed loop uses $U=\F[\hat v]$. By~\cite[Thm.~2]{nball}, if $v_0\in H^1_0(\Ball)$ and $\hat v_0=0$, there exist $M_{\mathrm{of}}\geq 1$ and $c_{\mathrm{of}}>0$ such that, with $\tilde v = v-\hat v$,
\begin{equation}
\|v(t)\|_{H^1(\Ball)}+\|\tilde v(t)\|_{H^1(\Ball)}
\leq M_{\mathrm{of}} e^{-c_{\mathrm{of}}t}\|v_0\|_{H^1(\Ball)}.
\label{eq:of-ball-decay}
\end{equation}
The constant in \eqref{eq:of-ball-decay} has been enlarged, if needed, to replace the estimate for $(v,\hat v)$ in~\cite{nball} by the equivalent estimate for $(v,\tilde v)$.

\subsection{Ball, Neumann actuation: a new explicit law}\label{sec:neumann}

The designs of~\cite{nball} are Dirichlet-actuated. A Neumann-actuated explicit law on the $n$-ball is new, to the best of our knowledge, yet it follows from the \emph{same} kernel \eqref{eq:kernel-bessel} at almost no cost, because the Volterra structure of \eqref{eq:transformation} does not involve the boundary condition at $r=R$: the kernel is determined by the PDE and by the conditions on the diagonal and at the origin, while the target boundary condition is \emph{enforced by the choice of control}. Differentiating \eqref{eq:transformation} in $r$ at $r=R$ and setting $\partial_r w_{\ell,m}(t,R)=0$ yields the flux control law. The target system is now \eqref{eq:target-pde} with homogeneous Neumann condition $\partial_r w|_{\partial\Ball}=0$, for which the energy estimate gives $\|w(t)\|\le e^{-ct}\|w(0)\|$; here the damping $c>0$ is essential rather than cosmetic, since the zero-reaction Neumann target has a constant mode that does not decay.

\begin{proposition}[Neumann-actuated ball backstepping]\label{prop:neumann}
Consider the plant \eqref{eq:ball-pde} with Neumann actuation $\partial_r v(t,x)=U^N(t,x)$ on $\partial\Ball$, and the modewise feedback
\begin{equation}\label{eq:neumann-law}
U^N_{\ell,m} = K_\ell(R,R)\, v_{\ell,m}(R) + \int_0^R \partial_r K_\ell(r,s)\big|_{r=R}\, v_{\ell,m}(s)\,ds,
\end{equation}
where $K_\ell(R,R)=-\mu R/2$ and, with $z=\sqrt{\mu(R^2-s^2)}$,
\begin{equation}\label{eq:neumann-kernel}
\partial_r K_\ell(R,s) = \mu s\left(\frac{s}{R}\right)^{\!\ell+n-2}\!\left[\frac{(\ell{+}n{-}2)}{R}\frac{I_1(z)}{z} - \mu R\,\frac{I_2(z)}{z^2}\right].
\end{equation}
Then for every $v_0\in L^2(\Ball)$ the closed loop has a unique weak solution and satisfies
\begin{equation}
\|v(t,\cdot)\|_{L^2(\Ball)} \leq C_{\mathcal T}\, e^{-ct}\|v_0\|_{L^2(\Ball)},\qquad t\ge 0.
\end{equation}
\end{proposition}

\begin{proof}
The transformation \eqref{eq:transformation} maps the plant equation to \eqref{eq:target-pde} modewise, irrespective of the outer boundary condition: the kernel equation and its conditions on $s=r$ and $s=0$ are those of~\cite{nball} and do not involve $r=R$. Differentiating \eqref{eq:transformation} in $r$ gives
\begin{multline*}
\partial_r w_{\ell,m}(R) = \partial_r v_{\ell,m}(R) - K_\ell(R,R)v_{\ell,m}(R) \\ - \int_0^R \partial_r K_\ell(R,s)v_{\ell,m}(s)\,ds,
\end{multline*}
so \eqref{eq:neumann-law} is equivalent to $\partial_r w_{\ell,m}(R)=0$. Formula \eqref{eq:neumann-kernel} follows by direct differentiation, using $\tfrac{d}{dz}(I_1(z)/z)=I_2(z)/z$ and $\partial_r z = \mu r/z$; the diagonal value follows from $I_1(z)/z\to 1/2$.

Multiplication of the Neumann target equation by $w$ and integration over $\Ball$ give
\begin{equation}\label{eq:neumann-energy}
\frac{1}{2}\frac{d}{dt}\|w(t)\|_{L^2(\Ball)}^2
=-\eps\|\nabla w(t)\|_{L^2(\Ball)}^2-c\|w(t)\|_{L^2(\Ball)}^2,
\end{equation}
so $\|w(t)\|_{L^2}\le e^{-ct}\|w_0\|_{L^2}$. With $v=(I+\Lop)w$ and $\Lop$ bounded on $L^2(\Ball)$~\cite[Prop.~1]{nball}, the estimate follows for every $v_0\in L^2(\Ball)$, as in Theorem~\ref{thm:ball-backstepping}.
\end{proof}

The Dirichlet and Neumann ball laws define different target-domain configurations. The Dirichlet target decays at rate $c+\eps\lambda_1(\Ball)$, while the Neumann target has the guaranteed rate $c$; their overshoot constants and boundary signals also differ. The physical actuation type on the irregular boundary is a separate choice when the extension proceeds through that boundary (Section~\ref{sec:coupled}).

\subsection{Radially-varying coefficients and the choice of center}\label{sec:radial}

When the reaction coefficient varies radially, $\lambda=\lambda(r)$, the constant-coefficient kernels \eqref{eq:kernel-bessel} no longer apply. The power-series designs of~\cite{scl2023,zhang2024} cover radial coefficients satisfying the analyticity and evenness assumptions stated there. Their kernels retain the modewise Volterra structure. The center of the target ball is then fixed by the coefficient: if the physical coefficient is $\lambda(|x-p|)$, the ball is centered at $p$ and the same coefficient is used in the virtual region.

\subsection{Rectangular target domains}\label{sec:square}

On a rectangle with actuation on one side and homogeneous conditions on the others, Fourier expansion in the tangential variable reduces the design to a family of one-dimensional kernels~\cite{krstic,bamieh2002distributed}, in the same closed Bessel form as \eqref{eq:kernel-bessel}. Per mode, both collocated observers (measurement on the actuated side) and anti-collocated observers (measurement on the opposite side) are classical~\cite{krstic}. Rectangular targets are the natural choice when part of the physical boundary consists of straight uncontrolled walls.

\section{The Domain-Extension Framework}\label{sec:coupled}

\subsection{The physical plant}\label{sec:plant}

The plant is a reaction-diffusion equation on a bounded Lipschitz domain $\Om\subset\R^n$, possibly non-convex, possibly multiply connected:
\begin{equation}\label{eq:plant}
u_t = \eps\Delta u + \lambda u, \qquad x \in \Om,\ t>0.
\end{equation}
Its boundary splits as $\partial\Om = \Gam_u \cup \Gam_c$: an uncontrolled part $\Gam_u$ (possibly empty) carrying homogeneous Dirichlet or Neumann conditions, and an actuated part $\Gam_c$, with actuation of Dirichlet or Neumann type. For output feedback, the collocated boundary measurement on $\Gam_c$ is available: the flux if the actuation is Dirichlet, the trace if it is Neumann. The objective is exponential stabilization of $u\equiv 0$.

\subsection{Target domains and compatibility}

\begin{definition}[Target-domain configuration]\label{def:target}
A \emph{target-domain configuration} for the plant \eqref{eq:plant} consists of
(i) a bounded Lipschitz target domain $\Dom$ containing $\Om$;
(ii) a partition $\partial \Dom = \Sigma_u \cup \Sigma_c$, with a Dirichlet or Neumann condition assigned to each boundary portion and, for output feedback, a measured set $\Sigma_m \subseteq \partial \Dom$;
(iii) a feedback law $\F$ for the dynamics \eqref{eq:plant} on $\Dom$ and, when needed, an observer. The resulting target closed loop is assumed to have a unique weak solution and to satisfy
\begin{equation}\label{eq:target-decay}
\|v(t)\|_{L^2(\Dom)} \le C_\Dom\, e^{-ct}\, \|v(0)\|_{L^2(\Dom)},\qquad t\ge 0.
\end{equation}
\end{definition}

Definition~\ref{def:target} is independent of backstepping. Any known design with the stated well-posedness and decay properties can be used. Backstepping supplies the closed-form configurations used here. The ball with the Dirichlet design of Section~\ref{sec:ball-recap} has $\Sigma_u=\emptyset$ and $\Sigma_c=\Sigma_m=\partial\Ball$. The same ball with Proposition~\ref{prop:neumann} gives a different configuration. A rectangle with one actuated side and homogeneous walls gives another one.

\begin{assumption}[Compatibility]\label{ass:compat}
The plant \eqref{eq:plant} and a target-domain configuration are \emph{compatible} if:
(i) $\overline\Om\subset\overline\Dom$, the extension region $\Omext=\Dom\setminus\overline\Om$ is a finite union of bounded Lipschitz domains, and $\Gam:=\partial\Om\cap\Dom$ is a relatively open Lipschitz interface; junctions between $\Gam$ and $\partial\Dom$ have surface measure zero;
(ii) the shared part $\Gam_u:=\partial\Om\cap\partial\Dom$ lies in $\Sigma_u$ and carries the same Dirichlet or Neumann condition in the plant and target configuration;
(iii) the extended part is exactly the actuated boundary, $\Gam_c=\Gam$, and the complementary collocated measurement is available there;
(iv) for output feedback, measurements required on $\Sigma_m\cap\Gam_u$ are supplied by physical sensors, while those on $\Sigma_m\setminus\Gam_u$ are computed from the virtual state as described in Section~\ref{sec:observer}.
\end{assumption}

The set $\Omext$ is the \emph{extension region}. It need not be connected: a hole in $\Om$ produces an enclosed component of $\Omext$. We write $n_\Om$ and $n_{\Omext}$ for the outward unit normals, defined almost everywhere on the Lipschitz boundaries, so $n_{\Omext}=-n_\Om$ almost everywhere on $\Gam$. For ball configurations we take $R>\max_{x\in\overline\Om}|x|$, hence $\Gam\Subset\Ball$. In shared-wall configurations, $\overline\Gam$ meets $\partial\Dom$ at the junction set allowed in Assumption~\ref{ass:compat}(i). Star-shaped domains are convenient for plotting and polar grids, but the theory does not use that property.

\begin{figure*}[t]
\centering
\includegraphics[width=0.95\textwidth]{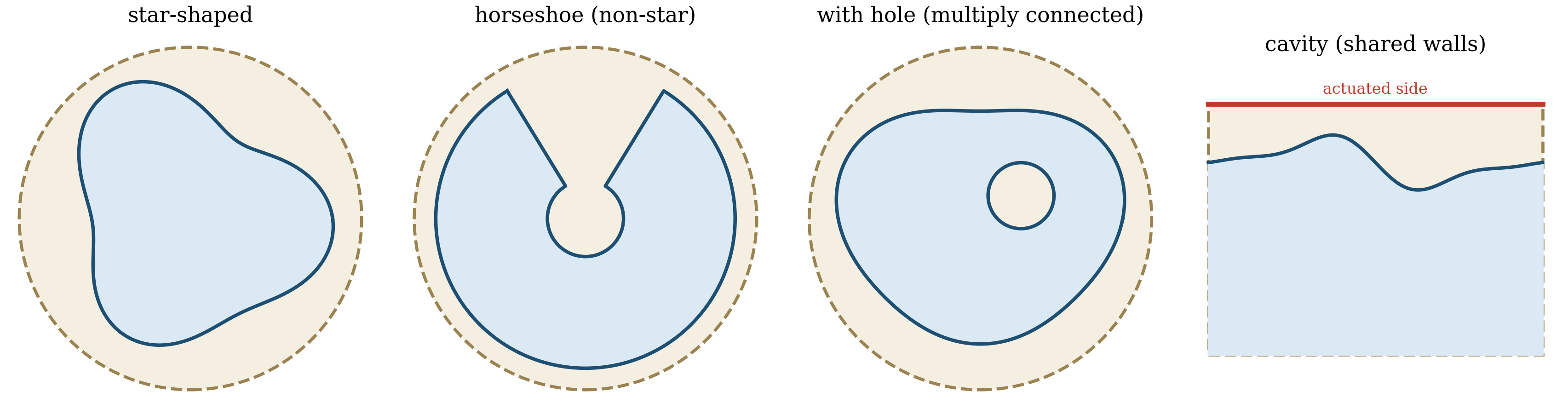}
\caption{The four physical domains $\Om$ (shaded) used in this paper, embedded in their target domains (dashed). The theory requires only that $\Om$ be Lipschitz and compatible with the target: convexity, star-shapedness and simple connectedness are nowhere assumed, as the horseshoe and the multiply connected example illustrate. In the cavity (rightmost), the straight walls are \emph{shared} with the rectangular target and only the top side of the rectangle is actuated.}
\label{fig:star-embedding}
\end{figure*}

\subsection{The coupled system}

Let $z$ denote the virtual state on $\Omext$. States on $\Om$ and $\Omext$ are concatenated by the operator $\concat: L^2(\Om)\times L^2(\Omext) \to L^2(\Dom)$,
\begin{equation}\label{eq:concatenated-state}
\concat(u,z)(x):=
\begin{cases}
u(x), & x\in \Om,\\
z(x), & x\in \Omext,
\end{cases}
\end{equation}
and we write $v=\concat(u,z)$ for the \emph{concatenated state}. The coupled real/virtual closed loop is, for $t>0$,
\begin{align}
u_t &= \eps \Delta u+\lambda u, && x\in \Om, \label{eq:u-eq}\\
z_t &= \eps \Delta z+\lambda z, && x\in \Omext, \label{eq:z-eq}\\
u &= z, && x\in \Gam, \label{eq:interface-dirichlet}\\
\partial_{n_\Om}u &= \partial_{n_\Om}z, && x\in \Gam, \label{eq:interface-flux}\\
z &= \F[\concat(u,z)], && x\in \Sigma_c, \label{eq:outer-feedback}
\end{align}
together with the homogeneous conditions on $\Gam_u$ (for $u$) and on $\Sigma_u\setminus\Gam_u$ (for $z$), and initial conditions $u(0,\cdot)=u_0\in L^2(\Om)$, $z(0,\cdot)=z_0\in L^2(\Omext)$. In \eqref{eq:outer-feedback}, $\F$ is the target design's feedback, written for Dirichlet actuation of the target; for the Neumann law of Proposition~\ref{prop:neumann} the condition reads $\partial_n z = \F^N[\concat(u,z)]$ on $\Sigma_c$. The interface conditions \eqref{eq:interface-dirichlet}--\eqref{eq:interface-flux} state that trace and flux are each continuous across $\Gam$: the same value, the same flux, seen from both sides.

\emph{The controller is dynamic, and the plant's boundary conditions fix the roles at the interface.} System \eqref{eq:u-eq}--\eqref{eq:outer-feedback} is a feedback interconnection of the plant with a dynamic boundary controller whose internal state is $z$: the controller simulates \eqref{eq:z-eq} in real time, fed on $\Sigma_c$ by the target design. The two interface conditions split into an actuation channel (imposed by the controller on the plant) and a sensing channel (fed by the plant to the controller), and the plant's actuation type decides the split:
\begin{itemize}
\item \emph{Dirichlet-actuated plant:} the plant receives $u|_\Gam = z|_\Gam$; the controller receives the measured flux $\partial_{n_\Om}u|_\Gam$ as Neumann data for $z$ on $\Gam$.
\item \emph{Neumann-actuated plant:} the plant receives $\partial_{n_\Om}u|_\Gam = \partial_{n_\Om}z|_\Gam$; the controller receives the measured trace $u|_\Gam$ as Dirichlet data for $z$ on $\Gam$.
\end{itemize}
Both assignments enforce \eqref{eq:interface-dirichlet}--\eqref{eq:interface-flux}, and both are driven by the same target design on $\Sigma_c$: the physical actuation type is decoupled from the actuation type of the target design. A Robin extension can be formulated by imposing one linear combination of trace and flux and measuring an independent combination. Its treatment requires the corresponding mixed-boundary weak formulation, so the results below are stated for Dirichlet and Neumann actuation. In both cases, the controller actuates and senses on all of $\Gam$.

\emph{Target-domain interpretation.} Classical backstepping transforms $u_t = \eps\Delta u + \lambda u$ into $w_t = \eps\Delta w - cw$ while preserving the spatial domain. The coupled system \eqref{eq:u-eq}--\eqref{eq:outer-feedback} preserves the plant dynamics while transforming the spatial domain from $\Om$ to~$\Dom$. The extension region carries the exact plant dynamics; nothing is approximated or penalized at the interface, whose only role is to make the concatenated state a genuine solution on the target domain. As the next section makes precise, the physical state is, for positive times, the restriction of that target closed loop, so $\Om$ evolves exactly as if it were a sub-region of the controlled target. The target geometry thus enters the plant's dynamics through the interface data; the plant is, in effect, controlled as a piece of a ball or a rectangle.

\begin{figure*}[t]
\centering
\includegraphics[width=0.92\textwidth]{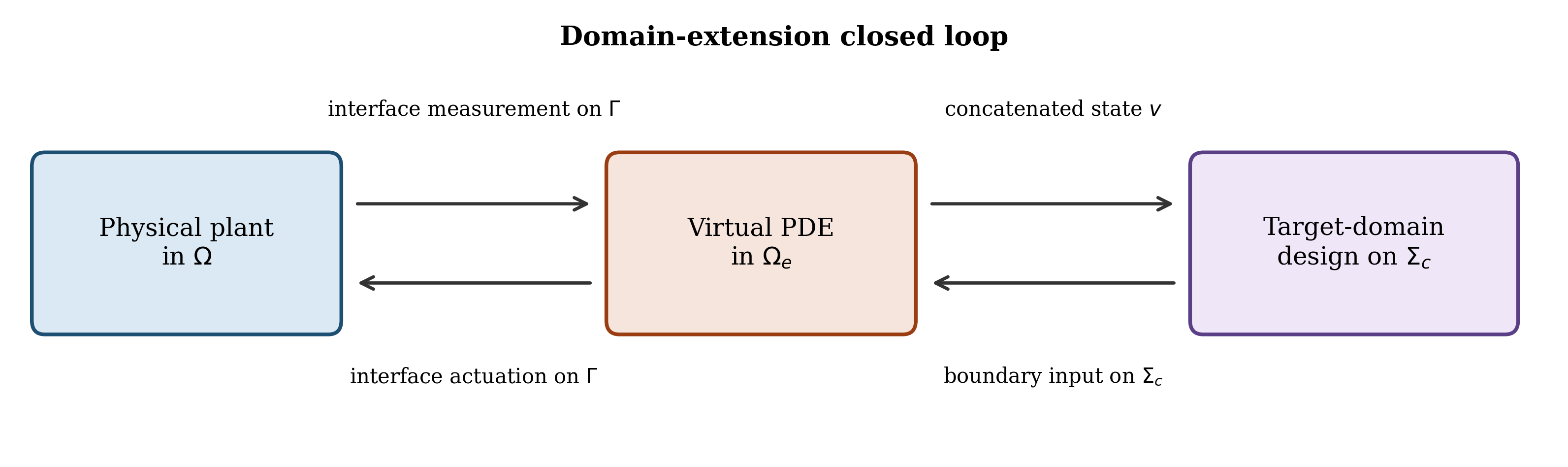}
\caption{Closed-loop information flow, drawn for a Dirichlet-actuated plant. The plant exports its interface measurement to the virtual domain, the virtual state returns the interface actuation, and the target design acts on the target boundary. For a Neumann-actuated plant the two interface arrows swap roles.}
\label{fig:workflow}
\end{figure*}

\subsection{Weak solutions}\label{sec:weak}

We use the standard trace operator on bounded Lipschitz domains; see, for example,~\cite[Ch.~18]{leoni}. The following elementary gluing fact is the only Sobolev result specific to the interface.

\begin{lemma}[Sobolev gluing]\label{lem:gluing}
Let Assumption~\ref{ass:compat}(i) hold. For $u\in H^1(\Om)$ and $z\in H^1(\Omext)$, the concatenation $\concat(u,z)$ belongs to $H^1(\Dom)$ if and only if the traces of $u$ and $z$ agree almost everywhere on $\Gam$. In that case,
\begin{equation}\label{eq:gluing-norm}
\|\concat(u,z)\|_{H^1(\Dom)}^2
=\|u\|_{H^1(\Om)}^2+\|z\|_{H^1(\Omext)}^2.
\end{equation}
\end{lemma}

\begin{proof}
Assume first that the traces agree. For any $\psi\in C_c^\infty(\Dom)$, integration by parts on $\Om$ and on each component of $\Omext$ shows that the interface terms in the distributional derivative of $\concat(u,z)$ cancel. Its weak gradient is therefore the concatenation of $\nabla u$ and $\nabla z$, which belongs to $L^2(\Dom)^n$ and gives \eqref{eq:gluing-norm}. The converse follows by restricting an $H^1(\Dom)$ function to the two Lipschitz subdomains: both one-sided traces are the trace of the same function on $\Gam$.
\end{proof}

Let $\Sigma_D$ be the union of the portions of $\partial\Dom$ where the target-domain configuration assigns a Dirichlet condition, controlled or homogeneous, and set
\begin{equation}\label{eq:test-space}
V_D:=\{\varphi\in H^1(\Dom):\ \varphi|_{\Sigma_D}=0\}.
\end{equation}
This choice leaves test functions free on homogeneous Neumann portions. To see the interface condition, first take smooth $(u,z)$ and $\varphi\in V_D$. Integration by parts on the two subdomains gives
\begin{align}
\int_{\Om} (u_t-\lambda u)\varphi\,dx
&= -\eps\!\int_{\Om} \nabla u\cdot\nabla\varphi\,dx
+\eps\!\int_{\Gam} (\partial_{n_\Om}u)\,\varphi\,dS, \label{eq:ibp-u}\\
\int_{\Omext} (z_t-\lambda z)\varphi\,dx
&= -\eps\!\int_{\Omext} \nabla z\cdot\nabla\varphi\,dx
-\eps\!\int_{\Gam} (\partial_{n_\Om}z)\,\varphi\,dS, \label{eq:ibp-z}
\end{align}
where homogeneous Neumann terms vanish, Dirichlet terms vanish through the test function, and the sign in \eqref{eq:ibp-z} uses $n_{\Omext}=-n_\Om$ on $\Gam$. Adding the identities leaves the interface term
\begin{equation}\label{eq:interface-cancel}
\eps\int_\Gam \bigl(\partial_{n_\Om}u-\partial_{n_\Om}z\bigr)\,\varphi\,dS,
\end{equation}
which vanishes under \eqref{eq:interface-flux}. Thus trace matching places the concatenated state in $H^1(\Dom)$, while the variational identity encodes flux matching without requiring a pointwise trace of the gradient.

\begin{definition}\label{def:weak}
A pair $(u,z)$ with
\begin{align*}
u &\in C([0,\infty);L^2(\Om))\cap L^2_{\mathrm{loc}}((0,\infty);H^1(\Om)),\\
z &\in C([0,\infty);L^2(\Omext))\cap L^2_{\mathrm{loc}}((0,\infty);H^1(\Omext)),
\end{align*}
is a \emph{weak solution} of the displayed Dirichlet-target system \eqref{eq:u-eq}--\eqref{eq:outer-feedback} with data $(u_0,z_0)$ if $u(0)=u_0$, $z(0)=z_0$, and for almost every $t>0$:
(i) the traces from both sides coincide on $\Gam$, $u(t)|_\Gam = z(t)|_\Gam$;
(ii) the Dirichlet conditions on $\partial\Dom$ hold in the trace sense, including $z|_{\Sigma_c} = \F[\concat(u,z)]$;
(iii) for every $\varphi\in V_D$,
\begin{multline}\label{eq:broken-weak}
\frac{d}{dt}\Bigl[\int_{\Om} u\varphi\,dx + \int_{\Omext} z\varphi\,dx\Bigr]
= -\eps\Bigl[\int_{\Om} \nabla u\cdot\nabla\varphi\,dx \\ + \int_{\Omext} \nabla z\cdot\nabla\varphi\,dx\Bigr]
+ \lambda\Bigl[\int_{\Om} u\varphi\,dx + \int_{\Omext} z\varphi\,dx\Bigr]
\end{multline}
in the sense of distributions on $(0,\infty)$.
\end{definition}

This is the usual variational formulation for a linear parabolic equation~\cite[Sec.~7.1]{evans}, written as two integrals. A weak solution of the target closed loop is defined by the same identity with one integral over $\Dom$. If the target actuation is Neumann, $\Sigma_c$ is excluded from $\Sigma_D$ and the right-hand side of \eqref{eq:broken-weak} contains the boundary term
\begin{equation}\label{eq:neumann-weak-term}
\eps\,\bigl\langle \F^N[\concat(u,z)],\varphi|_{\Sigma_c}\bigr\rangle,
\end{equation}
with the usual boundary duality. Homogeneous Neumann portions contribute zero. This convention covers the Dirichlet and Neumann configurations used below.

\section{Equivalence, Well-Posedness, and Stability Transfer}\label{sec:reduction}

\begin{proposition}[Equivalence with the target closed loop]\label{prop:equivalence}
Let the plant and a target-domain configuration be compatible. A pair $(u,z)$ is a weak solution of the coupled system with data $(u_0,z_0)$ if and only if $v=\concat(u,z)$ is a weak solution of the target closed loop on $\Dom$ with data $v_0=\concat(u_0,z_0)$.
\end{proposition}

\begin{proof}
($\Rightarrow$) Lemma~\ref{lem:gluing} gives $v(t)\in H^1(\Dom)$ for almost every $t>0$. The broken identity \eqref{eq:broken-weak} is then the target-domain variational identity, since its two volume integrals are the corresponding integrals over $\Dom$. The boundary conditions follow from item (ii). The same argument includes \eqref{eq:neumann-weak-term} for a Neumann target configuration.

($\Leftarrow$) Let $v$ be a weak solution on $\Dom$ and set $(u,z)=(v|_\Om, v|_{\Omext})$. Restriction preserves the regularity classes, the two one-sided traces of $v\in H^1(\Dom)$ on $\Gam$ coincide a.e., giving (i); (ii) holds by hypothesis; and \eqref{eq:broken-weak} is the weak form on $\Dom$ with the integral split over $\Om\cup\Omext$ ($\Gam$ has measure zero).
\end{proof}

\subsection{Well-posedness by restriction}

\begin{proposition}[Well-posedness]\label{prop:wellposed}
Under Assumption~\ref{ass:compat}, for every $u_0\in L^2(\Om)$ and $z_0\in L^2(\Omext)$, the coupled system \eqref{eq:u-eq}--\eqref{eq:outer-feedback} has a unique weak solution: the restriction $(u,z)=(v|_\Om,v|_{\Omext})$ of the target closed-loop solution from $v_0=\concat(u_0,z_0)$. For every $\tau>0$, $v$ is smooth on compact subsets of $\Dom\times[\tau,\infty)$. Hence the trace condition holds almost everywhere on $\Gam$, and the flux condition holds almost everywhere wherever the unit normal is defined. On a $C^1$ portion of the interface compactly contained in $\Dom$, both conditions hold classically for positive time.
\end{proposition}

\begin{proof}
Existence and uniqueness follow from Proposition~\ref{prop:equivalence} and Definition~\ref{def:target}. On any open set compactly contained in $\Dom$, the concatenated state satisfies $v_t-\eps\Delta v=\lambda v$ in distributions. The substitution $\tilde v=e^{-\lambda t}v$ removes the reaction term, so $\tilde v$ solves the heat equation $\tilde v_t=\eps\Delta\tilde v$; the heat equation smooths its data, hence $\tilde v$, and with it $v$, is smooth on compact subsets of $\Dom\times(0,\infty)$~\cite[Sec.~2.3.3]{evans}. The interface statements follow by restricting this single smooth function to the two sides; for a Lipschitz interface the normal is understood almost everywhere.
\end{proof}

For the ball designs of Section~\ref{sec:toolbox}, the well-posedness required by Definition~\ref{def:target}(iii) comes from the transformation: $v=(I+\Lop)w$ with $w$ the standard solution of the target system from $w_0=(I-\Kop)v_0$, and the transformations are bounded on $L^2$ and $H^1$~\cite[Prop.~1]{nball}, so $v$ inherits the regularity classes of Definition~\ref{def:weak} and the decay \eqref{eq:target-decay} with $C_\Dom=C_{\mathcal T}$. For a rectangular target $[0,L_x]\times[0,L_y]$, actuated on the side $y=L_y$, the state expands in sine modes, $v=\sum_{k\ge1} v_k(t,y)\sin(k\pi x/L_x)$, and each $v_k$ satisfies a one-dimensional reaction-diffusion equation with reaction $\lambda_k=\lambda-\eps(k\pi/L_x)^2$. Modes with $\lambda_k+c\le 0$ receive zero control and decay at rate $c$ or better with zero boundary data; each remaining mode is mapped to a damped one-dimensional target by the classical transformation with kernel $-\mu_k \eta\, I_1(\sqrt{\mu_k(y^2-\eta^2)})/\sqrt{\mu_k(y^2-\eta^2)}$, $\mu_k=(\lambda_k+c)/\eps$~\cite{krstic}. The kernel norms increase with $\mu_k$, and $\mu_k\le\mu_1$, so the per-mode transformation bounds are uniform in $k$, and summing the modal estimates by Parseval gives \eqref{eq:target-decay} and the regularity classes, exactly as on the ball.

The virtual initial state $z_0$ is the controller initialization. For the $L^2$ state-feedback result it can be chosen freely; $z_0=0$ is admissible even when the concatenated initial state jumps across $\Gam$. The weak target solution smooths for positive time, after which the interface conditions hold in the sense stated in Proposition~\ref{prop:wellposed}. Output feedback will use a compatible $H^1$ initialization instead, built from the known interface trace.

Restriction is an analysis device. In implementation the physical plant is not solved by the controller. Section~\ref{sec:numerics} therefore uses separate plant and controller solvers, coupled only through the interface data.

\subsection{Stability transfer}

\begin{theorem}[Stability transfer]\label{thm:main}
Let the plant \eqref{eq:plant} and a target-domain configuration be compatible. For every $u_0\in L^2(\Om)$ and $z_0 \in L^2(\Omext)$, the physical state of the coupled system \eqref{eq:u-eq}--\eqref{eq:outer-feedback} satisfies, for all $t\ge0$,
\begin{equation}
\|u(t,\cdot)\|_{L^2(\Om)} \leq C_\Dom\, e^{-ct}\bigl(\|u_0\|_{L^2(\Om)}^2+\|z_0\|_{L^2(\Omext)}^2\bigr)^{1/2}.
\label{eq:physical-stability}
\end{equation}
In particular, with the controller initialization $z_0=0$,
\begin{equation}
\|u(t,\cdot)\|_{L^2(\Om)} \leq C_\Dom\, e^{-ct}\|u_0\|_{L^2(\Om)}.
\label{eq:physical-stability-zero}
\end{equation}
\end{theorem}

\begin{proof}
By Proposition~\ref{prop:wellposed}, $u=v|_\Om$ with $v$ the target solution from $v_0=\concat(u_0,z_0)$, and $\|v_0\|_{L^2(\Dom)}^2=\|u_0\|_{L^2(\Om)}^2+\|z_0\|_{L^2(\Omext)}^2$. Then
$\|u(t)\|_{L^2(\Om)}\le\|v(t)\|_{L^2(\Dom)}\le C_\Dom e^{-ct}\|v_0\|_{L^2(\Dom)}$ by \eqref{eq:target-decay}.
\end{proof}

Restriction has norm one in $L^2$, so the decay rate and prefactor in \eqref{eq:physical-stability} are those of the target closed loop. With $z_0=0$, the initial norm also reduces to the physical norm. Hence, for a fixed target-domain configuration, $c$ and $C_\Dom$ are independent of the shape of the compatible physical domain embedded in $\Dom$. They are geometry-independent with respect to $\Om$, though changing the target domain or its design can change both constants.

The next corollary quantifies the physical actuation and sensing signals; it answers the natural implementation question of whether the boundary control generated through the virtual domain remains bounded.

\begin{corollary}[Decay of the interface signals]\label{cor:actuation}
Let $\Gam\Subset\Dom$. For every $\delta>0$ there exists $C_\delta>0$, depending on $\delta$, $\Gam$, and the target data, such that the solution of Proposition~\ref{prop:wellposed} satisfies, for all $t\ge\delta$,
\begin{equation}
\bigl\|u(t)|_\Gam\bigr\|_{L^2(\Gam)} + \bigl\|\partial_{n_\Om} u(t)|_\Gam\bigr\|_{L^2(\Gam)}
\le C_\delta\, e^{-ct}\,\|v_0\|_{L^2(\Dom)}.
\end{equation}
In shared-wall configurations the same estimate holds on every compact subset of $\Gam$ away from the junction points.
\end{corollary}

\begin{proof}
Cover $\Gam$ by finitely many Lipschitz coordinate neighborhoods whose closures lie in a larger open set $N'\Subset\Dom$. A standard interior estimate for the heat equation on the backward cylinder $N'\times(t-\delta/2,t)$, obtained by localizing the fundamental-solution representation~\cite[Sec.~2.3.1]{evans}, together with the target energy estimate \eqref{eq:target-decay}, gives
\[
\begin{aligned}
\|v(t)\|_{H^2(N)}
&\le C_{\delta,N}\|v\|_{L^2((t-\delta/2,t)\times\Dom)}\\
&\le C'_{\delta,N}e^{-ct}\|v_0\|_{L^2(\Dom)}
\end{aligned}
\]
for a smaller neighborhood $N$ of the interface: the interior estimate supplies the gain of derivatives, and \eqref{eq:target-decay} the exponential factor. The Lipschitz trace inequality gives $\|v|_\Gam\|_{L^2}\lesssim\|v\|_{H^1(N)}$. Applied to each component of $\nabla v\in H^1(N)^n$, it also gives $\|\partial_{n_\Om}v|_\Gam\|_{L^2}\lesssim\|v\|_{H^2(N)}$. The same argument works on compact subsets away from shared-boundary junctions.
\end{proof}

\emph{Implementing the target feedback.} The feedback $\F$ of the ball design can be evaluated in two ways, and both belong entirely to the controller. Modewise, one applies the explicit gains to finitely many angular modes of the concatenated state, and this truncation is exact in the following sense. The angular modes are modes of the \emph{target} ball, where they decouple; the irregular domain has no usable modal basis, and its eigenfunctions are precisely what this framework declines to compute. The $\ell$-th harmonic family with zero boundary data decays whenever $\eps\, j_{\nu(\ell),1}^2/R^2 > \lambda$, where $j_{\nu,1}$ is the first positive zero of the Bessel function $J_\nu$ and $\nu(\ell)=\ell+(n-2)/2$: a closed-form criterion on the ball, with no eigencomputation on $\Om$ anywhere. Controlling the finitely many modes that violate this inequality preserves exponential stability with rate $\min\{c,\ \min_{\ell>M}(\eps j_{\nu(\ell),1}^2/R^2-\lambda)\}$, by Proposition~\ref{prop:equivalence} applied to the truncated design; no spectral knowledge of the physical plant is involved. Alternatively, one evaluates the physical-space formula \eqref{eq:physical-law} directly, with no modal truncation at all.

\emph{Choice of center and radius (ball targets).} A smaller containing ball reduces the virtual volume and raises the first Dirichlet eigenvalue, so fewer target modes may require control. The kernel bounds also grow with the dimensionless size $\sqrt\mu R$. A natural choice is therefore the Chebyshev center of $\overline\Om$ and a radius slightly larger than its circumradius. For a radially varying coefficient, Section~\ref{sec:radial} fixes the center at the coefficient's symmetry point.

\emph{Shared boundaries.} When $\Gam_u\neq\emptyset$, the boundary conditions agree on the shared part by Assumption~\ref{ass:compat}(ii). Only $\Gam_c$ is actuated and sensed. The types must match on $\Gam_u$, where no virtual region is present to exchange trace and flux. Proposition~\ref{prop:neumann} and the rectangular designs of Section~\ref{sec:square} provide further choices for such configurations. At junctions with $\partial\Dom$, the interface relations are understood through the weak formulation; positive-time classical statements apply away from those points.

\emph{Robustness to model mismatch.} The actual plant remains one subsystem of the feedback interconnection. Coefficient uncertainty affects its matching with the virtual continuation, and the concatenated target equation then acquires piecewise distributed terms and interface disturbances. The target decay margin and standard parabolic energy estimates give reason to expect tolerance to sufficiently small coefficient and sensing errors. A quantitative threshold would also involve trace constants and the target feedback gains, especially their growth with $\sqrt\mu R$. We do not derive that threshold here; it is a natural extension of the nominal result.

\section{Output Feedback by Lifting the Target Observer}\label{sec:observer}

We now remove the full-state assumption. The state-feedback controller knows its virtual state but cannot evaluate the part of $\F[v]$ that depends on $u$ inside $\Om$. It therefore runs the observer on the target domain, where the required design is known. The observer is corrected by comparing the boundary data generated by the virtual state with its own prediction. Its estimate supplies $\F[\hat v]$ at the target boundary.

Formally, consider a Dirichlet-actuated plant; the physical output is the collocated interface measurement
\begin{equation}\label{eq:physical-measurement}
y_\Gam(t,\cdot)=\partial_{n_\Om} u(t,\cdot)\big|_\Gam ,
\end{equation}
which is precisely the datum the controller feeds to the virtual domain. The virtual domain then \emph{manufactures} the measurement that the target observer expects on the non-shared part of $\Sigma_m$: since $z=v$ on $\Omext$, which contains a neighborhood of $\Sigma_m\setminus\Gam_u$,
\begin{equation}\label{eq:measurement-identity}
\partial_{n_{\Dom}} z(t,\cdot)\big|_{\Sigma_m\setminus\Gam_u} = \partial_{n_{\Dom}} v(t,\cdot)\big|_{\Sigma_m\setminus\Gam_u},
\end{equation}
and the left-hand side is computable from the controller state. We call it the \emph{virtual measurement}: the physical interface data propagate through the virtual PDE to the target boundary. On $\Sigma_m\cap\Gam_u$, the target observer instead uses the physical measurement required by Assumption~\ref{ass:compat}(iv).

\subsection{Architecture}

The only quantity the state-feedback controller cannot compute without the full state is the feedback $\F[v]$, whose kernel integrates over all of $\Dom$, including the unknown physical region. Certainty equivalence therefore enters at the target boundary, and only there. The output-feedback controller carries two internal states, the virtual state $z$ and a target-domain observer $\hat v$. For the ball target ($\Sigma_u=\emptyset$, $\Sigma_c=\Sigma_m=\partial\Ball$), writing $\tilde y := \partial_{n_{\Ball}} z|_{\partial\Ball}-\partial_{n_{\Ball}}\hat v|_{\partial\Ball}$ for the (controller-internal) innovation,
\begin{align}
z_t &= \eps\Delta z + \lambda z, && x\in\Omext, \label{eq:of-z}\\
\partial_{n_\Om} z &= y_\Gam, && x \in \Gam, \label{eq:of-z-gamma}\\
z &= \F[\hat v], && x\in\partial\Ball, \label{eq:of-z-outer}\\[2pt]
\hat v_t &= \eps\Delta\hat v + \lambda\hat v + \mathcal{P}\bigl[\tilde y\bigr], && x\in \Ball, \label{eq:of-obs}\\
\hat v &= \F[\hat v], && x\in \partial\Ball, \label{eq:of-obs-bc}
\end{align}
with $\mathcal P$ the explicit output injection \eqref{eq:obs-gain}, and the plant is actuated with the \emph{true} virtual trace,
\begin{equation}\label{eq:of-actuation}
u\big|_\Gam = z\big|_\Gam .
\end{equation}
The target observer is initialized with $\hat v_0=0$. The virtual initial state is chosen to match the known interface trace. Since $\Gam\Subset\Ball$, the interface and the target boundary are separated. The trace theorem therefore gives a bounded lifting operator
\begin{equation}\label{eq:trace-lifting}
E_\Gam:H^{1/2}(\Gam)\longrightarrow H^1(\Omext),
\end{equation}
such that $(E_\Gam g)|_\Gam=g$ and $(E_\Gam g)|_{\partial\Ball}=0$~\cite[Ch.~18]{leoni}. Thus one may set $z_0=E_\Gam(u_0|_\Gam)$. The interface state is available at initialization: it is the imposed variable for Dirichlet actuation and the measured variable for Neumann actuation. If $u_0|_\Gam=0$, then $z_0=0$ is compatible. Equation~\eqref{eq:of-actuation} always uses the true controller state $z$, while the observer supplies the unavailable feedback value $\F[\hat v]$.

\subsection{Output-feedback transfer}

\begin{theorem}[Output-feedback transfer]\label{thm:output-feedback}
Let a Dirichlet-actuated plant and the Dirichlet ball configuration be compatible. Suppose $u_0\in H^1(\Om)$ and choose $z_0\in H^1(\Omext)$ so that
\begin{equation}\label{eq:of-compatible-data}
u_0|_\Gam=z_0|_\Gam,\qquad z_0|_{\partial\Ball}=0.
\end{equation}
Initialize the target observer with $\hat v_0=0$. Then the output-feedback loop \eqref{eq:of-z}--\eqref{eq:of-actuation} has a unique weak solution and
\begin{equation}\label{eq:of-transfer}
\begin{aligned}
\|u(t)\|_{H^1(\Om)}
&\leq M_{\mathrm{of}} e^{-c_{\mathrm{of}}t}\\[-2pt]
&\quad\times
\bigl(\|u_0\|_{H^1(\Om)}^2+\|z_0\|_{H^1(\Omext)}^2\bigr)^{1/2},
\qquad t\ge0.
\end{aligned}
\end{equation}
With $z_0=E_\Gam(u_0|_\Gam)$, there is a constant $C_\Gam$ depending only on the two domains such that
\begin{equation}\label{eq:of-transfer-lifted}
\|u(t)\|_{H^1(\Om)}
\leq C_\Gam M_{\mathrm{of}}e^{-c_{\mathrm{of}}t}\|u_0\|_{H^1(\Om)}.
\end{equation}
\end{theorem}

\begin{proof}
Condition \eqref{eq:of-compatible-data} and Lemma~\ref{lem:gluing} give $v_0=\concat(u_0,z_0)\in H^1(\Ball)$. Its outer trace is zero, hence $v_0\in H^1_0(\Ball)$. The actuation \eqref{eq:of-actuation} and measurement \eqref{eq:of-z-gamma} impose the two interface conditions, so Proposition~\ref{prop:equivalence} identifies $v=\concat(u,z)$ with the ball plant driven by $\F[\hat v]$. Identity \eqref{eq:measurement-identity} supplies the ball observer with its true boundary measurement. The pair $(v,\hat v)$ is therefore the output-feedback loop of Section~\ref{sec:ball-recap}, and \eqref{eq:of-ball-decay}, restriction, and \eqref{eq:gluing-norm} give \eqref{eq:of-transfer}.

For the lifted choice, the trace and lifting estimates give
\[
\|z_0\|_{H^1(\Omext)}
\le C_E\|u_0|_\Gam\|_{H^{1/2}(\Gam)}
\le C_EC_{\rm tr}\|u_0\|_{H^1(\Om)}.
\]
Substitution in \eqref{eq:of-transfer} proves \eqref{eq:of-transfer-lifted}.
\end{proof}

\emph{Sensing topology.} The lifted design is \emph{collocated on extended boundaries by construction}: the interface through which the domain is extended carries the actuation and the measurement at the same location, and this collocated pair is structurally necessary. Without the interface measurement, the virtual domain cannot be synchronized to the plant, and no information about the plant reaches the controller at all. Any further measurements must be those of the target design itself: on shared boundaries they are physical measurements (the plant boundary coincides there with the target boundary, so a wall sensor reads the concatenated state directly, and a target observer using that boundary, such as the anti-collocated rectangle observer of Section~\ref{sec:square}, lifts verbatim); on the remaining target boundary they are virtual, computed from $z$. Anti-collocated configurations \emph{through the extension}, with sensing only away from the extended boundary, are not covered by the lifting argument and remain open.

The target observer is not redesigned for the irregular geometry, and $\hat u=\hat v|_\Om$ estimates the physical state. For a Neumann-actuated plant, the interface trace is measured and the flux is assigned; the same concatenation argument applies once a compatible target observer is specified. The decay rate in Theorem~\ref{thm:output-feedback} is the target rate. Estimate \eqref{eq:of-transfer} keeps the target prefactor when written in terms of the concatenated initial state, while \eqref{eq:of-transfer-lifted} includes the norm of the geometric lifting.

For a sufficiently smooth interface and $H^2$ initial data, a standard higher-order trace lifting can also match the initial normal derivative~\cite[Ch.~18]{leoni}. This uses the available state and flux data at the interface. It removes the initial flux mismatch but is not needed for the $H^1$ result above.

\section{Numerical Experiments}\label{sec:numerics}

\subsection{Setup and implementation}

We test the framework on the four geometries of Fig.~\ref{fig:star-embedding}: a star-shaped domain with polar boundary $\rho(\theta) = 0.95(1 + 0.18\cos 3\theta - 0.08\sin 2\theta)$; a horseshoe (a thick annular sector with a $63^\circ$ opening), which is star-shaped with respect to no point; a multiply connected blob with an off-center circular hole, whose extension region has two components; and, for the shared-wall configuration, an irregular cavity in the rectangle $[0,2]\times[0,1.5]$, an irregular relative of the piano-shaped domain considered in~\cite{VK-Milano}. The disk-target geometries are embedded in $\Ball$ with $R=1.35$. A single parameter set is used throughout: $\eps = 0.05$, $\lambda = 0.85$, $c = 1.0$, so $\mu = 37$. The choice of $\lambda$ makes every physical plant open-loop unstable: the numerically computed first Dirichlet eigenvalues give $\eps\lambda_1(\Om) = 0.348$ (star), $0.637$ (horseshoe), $0.492$ (holed), and $0.518$ (cavity), all below $\lambda$. (These eigenvalues certify the instability of the uncontrolled baselines; the design itself never uses them.) All simulations are implemented in Python/SciPy~\cite{scipy} with sparse LU factorizations~\cite{superlu} and standard finite-volume and finite-difference discretizations~\cite{leveque-fdm}, on Cartesian grids of $201\times 201$ points, with $M=12$ angular modes and implicit Euler steps $\Delta t = 2.5\cdot 10^{-3}$ over the horizon $T=10$ ($T=14$ for output feedback).

Following the implementation philosophy of Section~\ref{sec:reduction}, the plant is treated as a physical system: it is discretized on its own index set, with its own factorization, and the controller (the virtual domain, plus the observer in the output-feedback experiments) is a separate solver. At each time step, two block Gauss--Seidel transmission sweeps exchange the interface data, a standard partitioned treatment of transmission problems in domain decomposition~\cite{quarteroni-valli}. The output-feedback run uses four sweeps. The target feedback is applied on $\Sigma_c$ from projected modal coefficients of the concatenated data. The online burden is therefore of the same nature as in a full-order observer implementation: sparse time stepping for parabolic PDEs. Output feedback couples the virtual continuation and observer through finitely many modal channels, so they are advanced as one block system rather than through a new online design calculation. As an independent check of Proposition~\ref{prop:equivalence} and Theorem~\ref{thm:main}, a second solver computes the target closed loop with radial finite volumes for each angular mode and restricts it to $\Om$. The two computations use different spatial discretizations and evolution solvers. Their discrete per-mode closed-loop and observer-error operators reproduce the theoretical spectral abscissas of Section~\ref{sec:toolbox} to three decimal places.

\subsection{State feedback}

Fig.~\ref{fig:sf-energy} shows the physical energies for the three disk-target geometries, against the uncontrolled baselines. Each uncontrolled plant grows at the rate predicted by its first eigenvalue, while the closed loop decays at the rate set on the target domain, uniformly across the three geometries, as Theorem~\ref{thm:main} predicts. The restriction and partitioned solutions track each other closely; the residual differences concentrate on the interfaces least resolved by the Cartesian mask and leave the energy decay unaffected. Refining the grid, time step, and mode count shrinks those differences and changes the final physical energy by under one percent. The figure also displays the physical actuation energy $\|u|_\Gam\|^2_{L^2(\Gam)}$, which decays after the backstepping transient, as Corollary~\ref{cor:actuation} asserts. Fig.~\ref{fig:sf-3d} shows the spatial decay on the three geometries.

\begin{figure*}[t]
\centering
\includegraphics[width=0.95\textwidth]{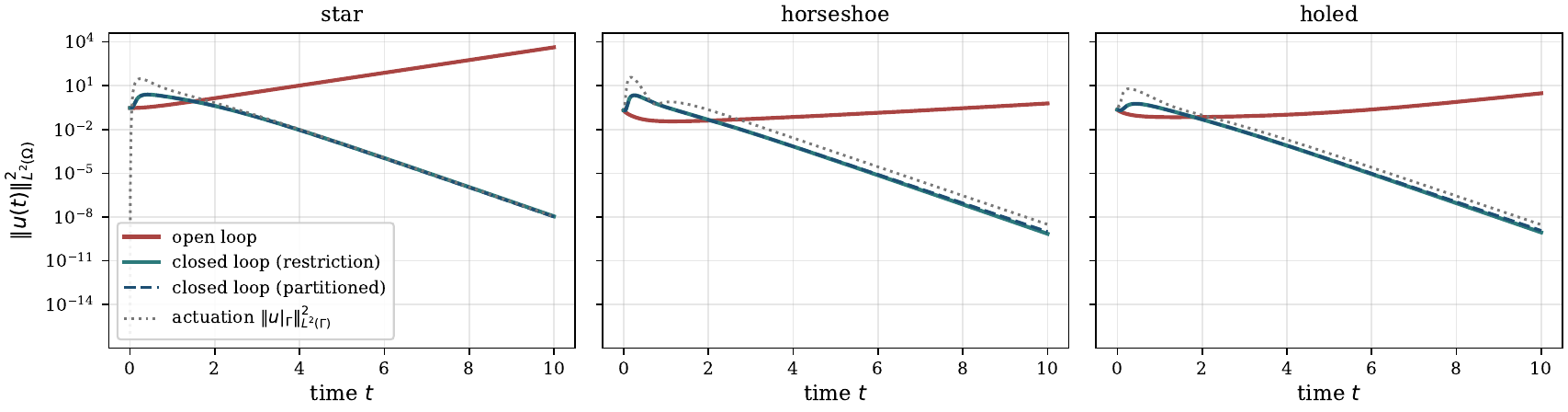}
\caption{State feedback on the three disk-target geometries: uncontrolled plant (red), closed loop via restriction (teal), closed loop via the partitioned plant/controller implementation (dashed blue), and physical actuation energy $\|u|_\Gam\|^2_{L^2(\Gam)}$ (dotted gray), which decays exponentially as asserted by Corollary~\ref{cor:actuation}. All energies are squared $L^2$ norms.}
\label{fig:sf-energy}
\end{figure*}

\begin{figure*}[t]
\centering
\includegraphics[width=0.97\textwidth]{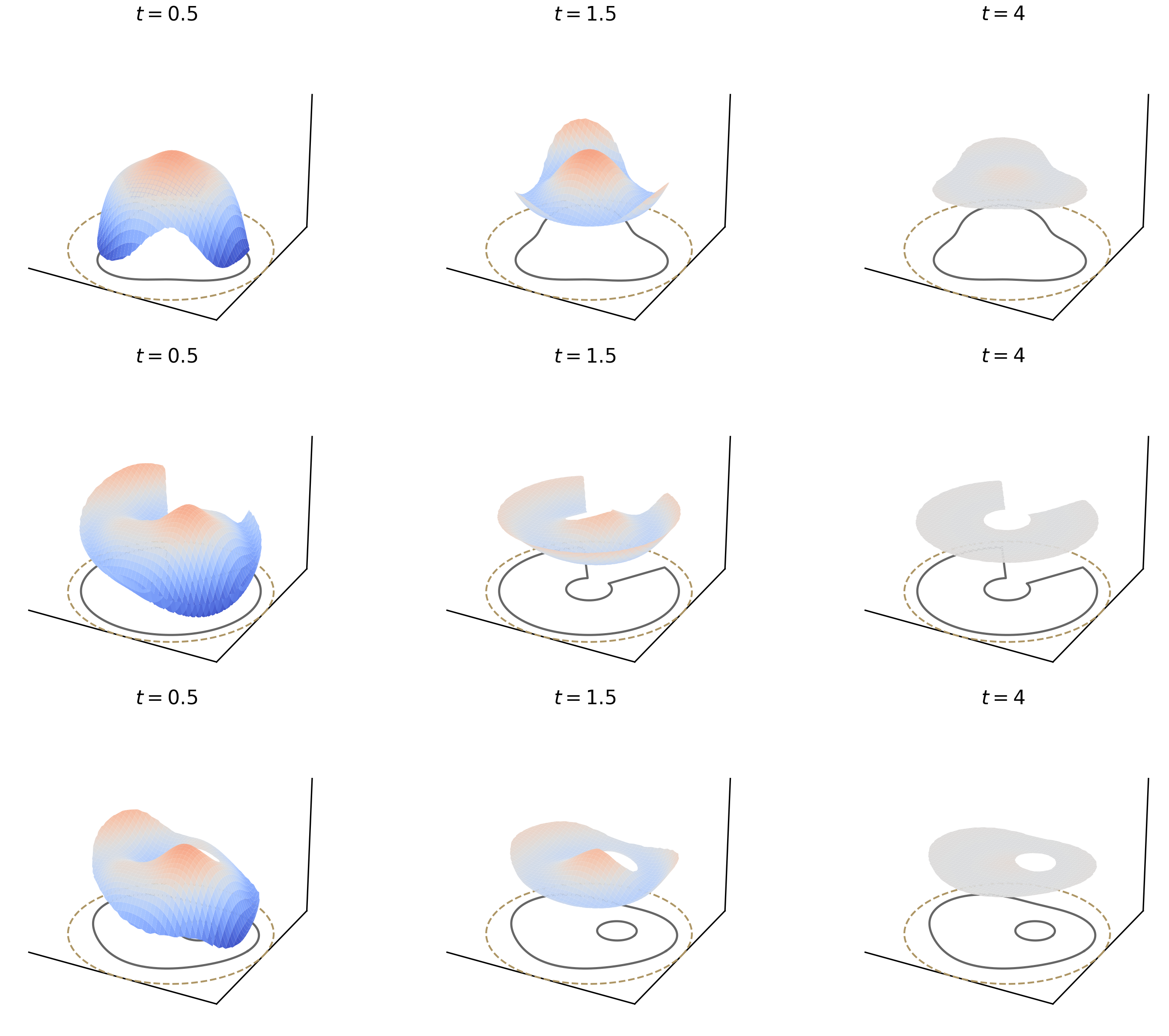}
\caption{Closed-loop state on the physical domains (partitioned implementation) at $t=0.5$, $1.5$, $4$; rows: star, horseshoe, multiply connected domain. The vertical scale is fixed per row, so the flattening of the surfaces displays the decay; the solid curve at the base is $\Gam$ and the dashed circle is $\partial\Ball$. Snapshots start at $t=0.5$, past the initial transient.}
\label{fig:sf-3d}
\end{figure*}

\subsection{Output feedback}

The output-feedback experiment implements the architecture of Section~\ref{sec:observer} with zero prior knowledge in the controller: $z_0=0$ and $\hat v_0=0$. The plant initial state is a smooth bump compactly supported inside $\Om$, so its trace on $\Gam$ vanishes and the zero virtual initialization satisfies the compatibility condition in Theorem~\ref{thm:output-feedback}. Sensing is used only on $\Gam$, and certainty equivalence enters only at $\partial\Ball$. The controller carries two internal simulations, the virtual domain on the extension region and the observer on the full target domain. The kernel scale is large for these parameters: $I_0(\sqrt\mu R)\approx 521$. The observer must therefore be discretized consistently with the virtual domain (here, on the same Cartesian grid), since the innovation compares the virtual measurement with the observer's prediction. The two controller states are integrated as one implicit system. Their couplings use $2(2M{+}1)$ modal channels, so one step costs three sparse solves and a small dense solve. The plant--controller exchange remains partitioned; we use four transmission iterations per step and $\Delta t \le 5\cdot 10^{-3}$.

Fig.~\ref{fig:output-feedback} shows the result on the star domain. The plant energy first rises while the observer locks on, the price of zero prior knowledge, then falls by roughly eight orders of magnitude, its tail steepening toward the state-feedback rate as the estimate converges; the observer-error energy decays monotonically, by about twelve orders over the horizon. Fig.~\ref{fig:of-3d} shows the same transient in space: at $t=0.5$ the error is essentially the unknown physical state inside $\Gam$, and it drains through the interface measurement well before the plant settles.

\begin{figure*}[t]
\centering
\includegraphics[width=0.85\textwidth]{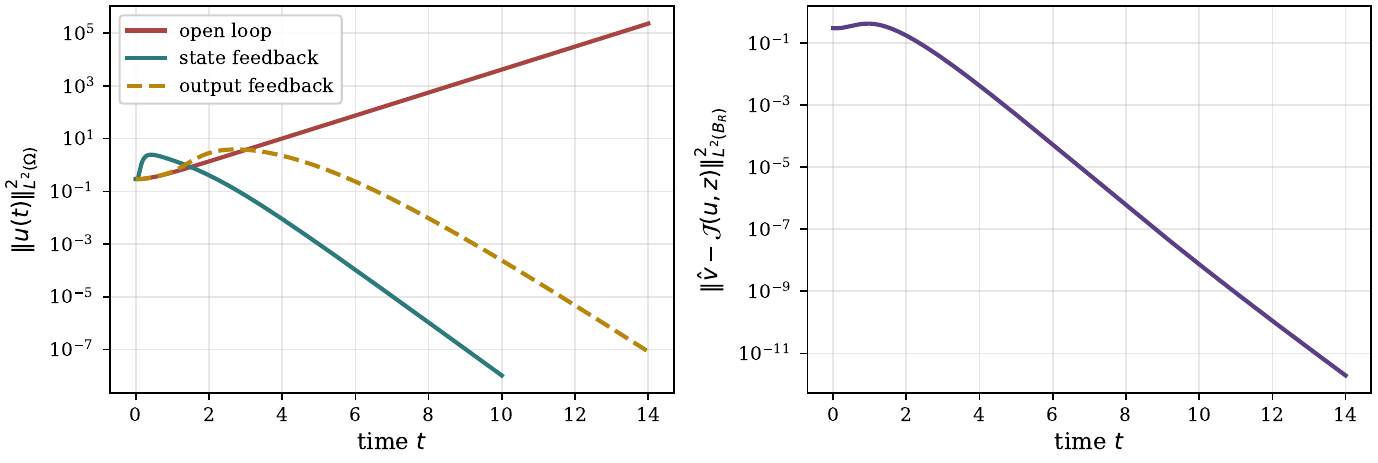}
\caption{Output feedback on the star domain, with $z_0=\hat v_0=0$ and a compatible compactly supported plant state. Left: physical energy under output feedback (dashed gold) vs.\ state feedback (teal) and the uncontrolled plant (red). Right: observer error energy on the ball.}
\label{fig:output-feedback}
\end{figure*}

\begin{figure*}[t]
\centering
\includegraphics[width=0.97\textwidth]{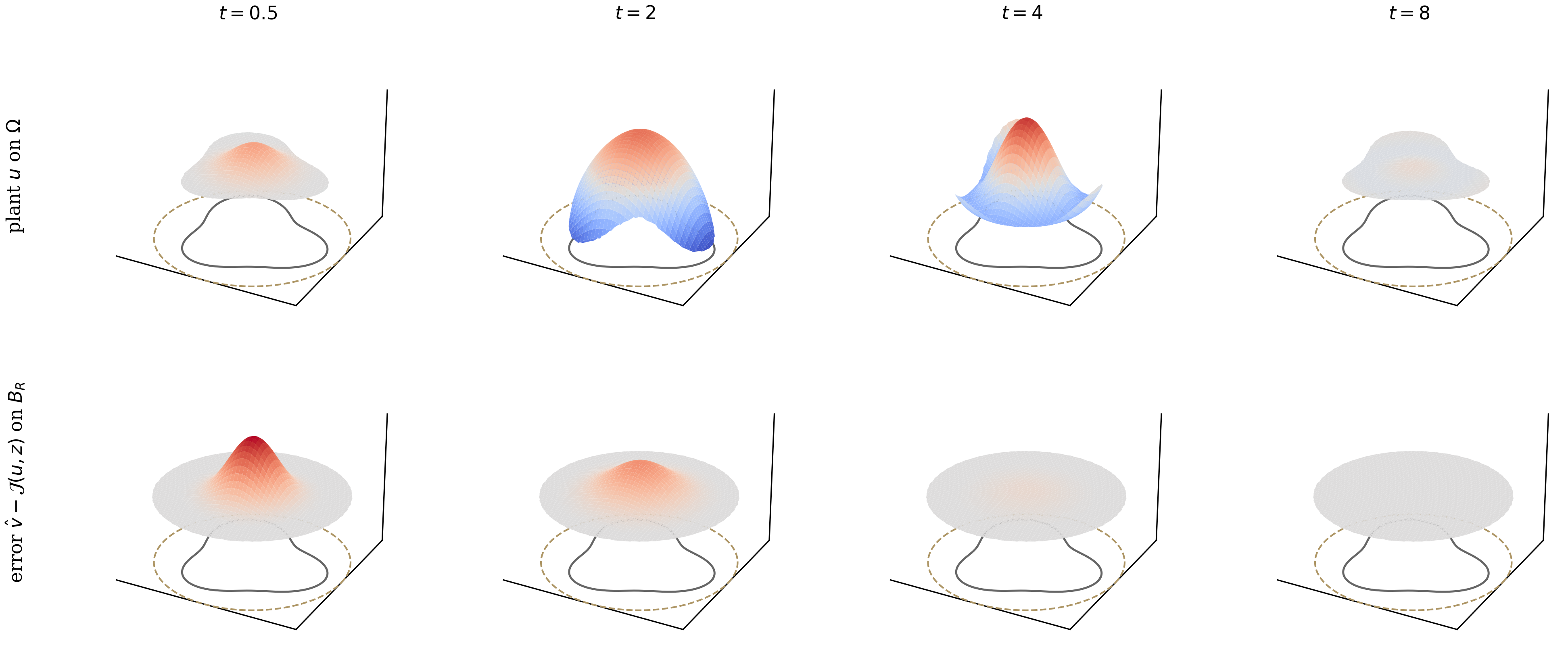}
\caption{Output feedback in space. Top row: plant state $u$ on $\Om$. Bottom row: observer error $\hat v - \concat(u,z)$ on the full target disk. The error starts as the unknown physical state supported inside $\Gam$ and decays through the interface measurement. Vertical scales are fixed per row.}
\label{fig:of-3d}
\end{figure*}

\subsection{Rectangular target with shared walls}

The cavity is controlled through the rectangular target $[0,2]\times[0,1.5]$: the flat bottom and side walls of the cavity are shared, uncontrolled, homogeneous-Dirichlet boundaries; the design-level input acts on the top side of the rectangle, which is virtual; and the physical actuation transmitted to the cavity is the trace on the wavy interface $\Gam$. The thin virtual strip keeps the target close to the physical domain, reducing both virtual volume and gain size. The design is the Fourier counterpart of the ball designs: sine modes in $x$, one-dimensional kernels in $y$ per mode, with only the modes satisfying $\lambda - \eps(k\pi/L_x)^2 + c > 0$ actuated; the rest are open-loop stable, by the same truncation argument as on the ball. Fig.~\ref{fig:cavity-3d} shows the closed-loop evolution, with the actuation wave entering from the virtual strip at $t=0.5$--$1$. The uncontrolled energy grows at the rate its first eigenvalue predicts, while the closed loop drives it down by more than eight orders of magnitude over the horizon.

\begin{figure*}[t]
\centering
\includegraphics[width=0.92\textwidth]{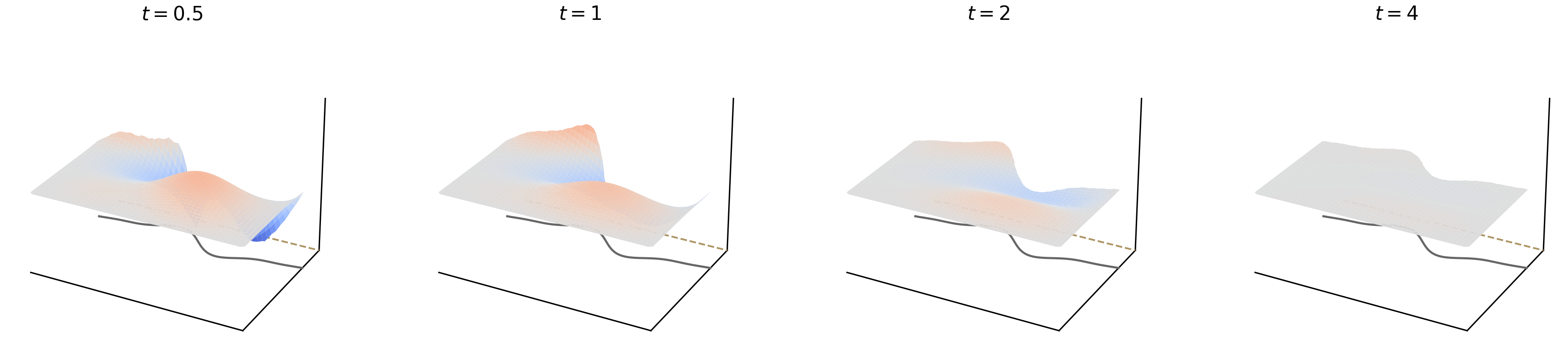}
\caption{Shared-wall cavity controlled through the rectangular target $[0,2]\times[0,1.5]$ (actuation on the top side only) at $t=0.5$, $1$, $2$, $4$. The cavity walls are uncontrolled and shared with the target; the wavy interface $\Gam$ receives its Dirichlet data from the thin virtual strip above it, through which the actuation wave is seen entering in the first two panels. The vertical scale is common to all panels.}
\label{fig:cavity-3d}
\end{figure*}

\section{Conclusions}\label{sec:conclusions}

This paper develops a domain-extension framework for boundary stabilization of reaction-diffusion equations on bounded Lipschitz domains. The transfer argument accepts any compatible target-domain configuration. In state feedback, fixing that configuration makes the decay rate and prefactor independent of the shape of the embedded physical domain. Output feedback preserves the target estimate when it is written in terms of the concatenated initial state; expressing it only through the physical $H^1$ norm introduces the lifting constant. A ball gives a full-boundary construction for any bounded domain, while rectangular targets permit shared straight walls and partial actuation. For the constant-coefficient cases treated here, the offline gains are closed form. The paper also derives a Neumann-actuated ball law and lifts the ball observer under compatible $H^1$ initialization.

The controller solves the virtual PDE online, but this is the same kind of operation already required by full-order PDE observers. It should therefore be compared with observer-based PDE control, not with static finite-dimensional feedback. In state feedback the controller evolves one virtual PDE; in output feedback it evolves the virtual continuation and the target observer, coupled as a single block system. The experiments use separate plant and controller solvers and quantify their agreement with the target restriction. Model mismatch deserves a separate analysis: the stable target supplies a decay margin, but the admissible coefficient and sensing errors should depend on trace constants, target size, and gain magnitude.

The same transfer proof applies in higher dimensions when the target design and the geometric assumptions are available. Other open points include sensing only on uncontrolled shared boundaries, interfaces outside the Lipschitz setting, richer canonical targets such as annular sectors that would hug domains like the horseshoe more tightly than a disk, and target configurations for hyperbolic systems~\cite{belhadjoudja-hyp1,belhadjoudja-hyp2}. The virtual PDE can also be viewed as an interface-data-to-boundary-data operator. Approximating that operator offline, for example with the neural-operator ideas of~\cite{bhan-no}, may reduce the online cost, but its effect on closed-loop stability would have to be quantified.


\end{document}